\documentclass{amsart}
\usepackage{eurosym}
\usepackage{amssymb}
\usepackage{amsfonts}
\usepackage{amsmath}
\usepackage{bbm}
\newtheorem{theorem}{Theorem}[section]
\newtheorem{proposition}{Proposition}[section]

\theoremstyle{definition}

\theoremstyle{remark}
\newtheorem{remark}[theorem]{Remark}

\theoremstyle{corollary}

\numberwithin{equation}{section}


\begin{document}
\title[Singular quasilinear elliptic problems]{Sub-supersolutions method for
singular quasilinear systems involving gradient terms}
\author[A. Moussaoui]{Abdelkrim Moussaoui}
\address{Applied Mathematics Laboratory, Faculty of Exact Sciences \\
and \\
Biology Department, Faculty of Natural and Life Sciences \\
A. Mira Bejaia University, Targa Ouzemour, 06000 Bejaia, Algeria}
\email{abdelkrim.moussaoui@univ-bejaia.dz}

\begin{abstract}
Existence, regularity and location of solutions to quasilinear singular
elliptic systems with general gradient dependence are established developing
a method of sub-supersolution. The abstract theorems involving
sub-supersolutions are applied to prove the existence of positive solutions
for convective and absorption singular systems.
\end{abstract}

\maketitle

\section{Introduction}

We deal with the following quasilinear elliptic system 
\begin{equation*}
(\mathrm{P})\qquad \left\{ 
\begin{array}{l}
-\Delta _{p_{1}}u_{1}=f_{1}(x,u_{1},u_{2},\nabla u_{1},\nabla u_{2})\text{
in }\Omega \\ 
-\Delta _{p_{2}}u_{2}=f_{2}(x,u_{1},u_{2},\nabla u_{2},\nabla u_{2})\text{
in }\Omega \\ 
u_{1},u_{2}>0\text{ in }\Omega \\ 
u_{1},u_{2}=0\text{ \ \ \ on }\partial \Omega%
\end{array}%
\right.
\end{equation*}%
where $\Omega \subset 
\mathbb{R}
^{N}$ $\left( N\geq 2\right) $ is a bounded domain with a smooth boundary $%
\partial \Omega $ and $\Delta _{p_{i}}$ denote the $p_{i}$-Laplacian
differential operator on $W_{0}^{1,p_{i}}(\Omega )$ with $1<p_{i}<N,$
defined by: $\Delta _{p_{i}}w_{i}=div(|\nabla w_{i}|^{p_{i}-2}\nabla w_{i})$
for $w_{i}\in W_{0}^{1,p_{i}}(\Omega )$, $i=1,2.$

A solution of $(\mathrm{P})$ is understood in the weak sense, that is a pair 
$(u_{1},u_{2})\in W_{0}^{1,p_{1}}(\Omega )\times W_{0}^{1,p_{2}}(\Omega )$
satisfying 
\begin{equation}
\int_{\Omega }|\nabla u_{i}|^{p_{i}-2}\nabla u_{i}\nabla \varphi _{i}\text{ }%
\mathrm{d}x=\int_{\Omega }f_{i}(x,u_{1},u_{2},\nabla u_{1},\nabla
u_{2})\varphi _{i}\text{ }\mathrm{d}x,  \label{10}
\end{equation}%
for all $\varphi _{i}\in W_{0}^{1,p_{i}}(\Omega )$. The nonlinear terms $%
f_{i}:\Omega \times \mathbb{R}^{2}\times 
\mathbb{R}
^{2N}\rightarrow 
\mathbb{R}
$ are Carath\'{e}odory functions, that is, $f_{i}(\cdot ,s_{1},s_{2},\xi
_{1},\xi _{2})$ is Lebesgue measurable for every $(s_{1},s_{2},\xi _{1},\xi
_{2})\in \mathbb{R}^{2}\times \mathbb{R}^{2N}$ and $f_{i}(x,\cdot ,\cdot
,\cdot ,\cdot )$ is continuous for a.e. $x\in \Omega ,$ $i=1,2.$

Problem $\left( \mathrm{P}\right) $ incorporates nonlinear terms $f_{1}$ and 
$f_{2}$ where the solution, its gradient, and spatial variables are
involved. This leads to a complex interplay of phenomena, mostly produced by
processes of reaction, absorption, and convection. These phenomena, which
could be deeply intertwined in the context of $\left( \mathrm{P}\right) $,
show interesting features. They are fundamental components in mathematical
models describing a wide range of nonlinear processes in engineering and
natural systems. From biological heat transfer and geological
thermoconvective motions to chemical reactions and global climate energy
balance \cite{K, DT}, problem $\left( \mathrm{P}\right) $ manifests itself
in numerous contexts. Furthermore, the convective-absorption problem can be
related to nonlinear Fokker-Planck equations and generalized viscous Burgers
equations, extending its relevance to plasma physics, astrophysics,
hydrodynamics, and neurophysics \cite{F}.

Convective and absorption problems have attracted significant interest,
resulting in a substantial number of research papers. A significant part is
carried on scalar case. Due to the sheer volume of publications, it is
impractical to quote them all. Among these, we refer to \cite{ACF, AM, CLMous,
DM, ElMiri, GMM, MMZ, MS} and there references. These problems present a
significant challenge due to the presence of gradient terms, effectively
precluding the application of standard variational methods. The absence of a
variational structure renders these problems unsuitable for techniques that
rely on minimizing functionals. Furthermore, even topological methods, such
as the sub-supersolution approach (see \cite{CMo, MVV}) and fixed-point
theorems (see \cite{CLMous, DM, MMZ}), struggle to provide solutions
directly. Their implementation requires controlling both the solution itself
and its gradient, a condition often difficult to satisfy. This difficulty is
particularly acute in the context of quasilinear systems $(\mathrm{P})$,
where constructing sub-supersolutions and establishing comparison properties
becomes significantly more complex. The inherent structural hurdle lies in
the inability to derive a priori relationships between the gradients of
comparable functions. This is an important fact when dealing with systems
involving gradient terms, which makes the application of topological methods
particularly difficult. This issue is addressed in \cite{CMo, MVV}, where
theorems involving sub-supersolutions have been proved using the theory of
pseudomonotone operators. In fact, even though results providing localized
solutions between sub- and supersolutions have been shown, their
applicability remains narrowly constrained. This limitation stems from the
reliance on conditions that effectively neutralize or minimize the impact of
gradients, as exemplified by assumptions (H2), (H3) in \cite{CMo} and
(H1)-(H3) in \cite{MVV}. Thence, the construction of sub-supersolutions for
comprehensive range of systems incorporating gradient terms remains an open
question.

Another important feature in the study of problem $(\mathrm{P})$ is that the
nonlinearities may exhibit singularities as the variables $u_{1}$ and $u_{2}$
approach zero. Their presence has a significant impact on both the structure
of $(\mathrm{P})$ and the properties of its solutions, which is an important
issue to be addressed in the present work. Actually, singularities are
involved in a wide range of important elliptic problems which have been
studied extensively in recent years, see for instance \cite{AM1, AMT, DM,
DM1, DM2, KM, MMM, MM3, MM2, MM1}.

Our main purpose is to establish the existence and regularity of solutions
for quasilinear system $(\mathrm{P})$ under appropriate growth conditions.
Our approach relies to sub-supersolutions technique we develop in section %
\ref{S2} for systems $(\mathrm{P})$ involving gradient terms and
singularities. For this purpose, taking advantage of Young and Hardy-Sobolev
inequalities on one hand, \cite[Theorem 2.1]{DM} on the other, we provide an
a priori estimate on the gradient of solutions in $L^{p_{1}}(\Omega )\times
L^{p_{2}}(\Omega )$ and $L^{\infty }(\Omega )\times L^{\infty }(\Omega )$,
respectively. This established a control on solutions and their gradients,
and thus made it possible to apply Schauder's fixed point theorem. Thus, a
weak solution for $(\mathrm{P})$ is obtained and is located in the rectangle
formed by the sub-supersolutions and acquires their sign and boundedness
properties. This result is stated in Theorem \ref{T1}. Another significant
feature of our existence result concerns the regularity part. At this point,
strengthening assumption on the vector field $(f_{1},f_{2})$ to be specified
later, Theorem \ref{T2} shows that solutions of singular systems $(\mathrm{P}%
)$ are smooth. It is worth pointing out that our argument is substantially
different from that developed in \cite{CMo} and \cite{MVV} where theory of
pseudomonotone operators is relevant. Moreover, it is noteworthy that our
investigation, if only in two ways, enhances that previously conducted in 
\cite{CMo} and \cite{MVV}. Firstly, the hypotheses under consideration
encompass those delineated in the aforementioned papers. Secondly, the
results of the current study enable the identification of smooth solutions
and extend their applicability to singular problems. Additionally, we should
note that our results can be extended straightforwardly to systems of more
than two equations. For the sake of simplicity, we restricted our
consideration to system $(\mathrm{P})$.

The applicability of the developed sub-supersolutions technique is provided
in Section \ref{S3} where systems involving convection and absorption terms
are considered. Relevant sub-supersolution pair for each system is
constructed through suitable functions, with an adjustment of adequate
constants, and a priori estimate on their gradient. The resulting
sub-supersolutions produce a rectangle that localizes a solution of the
convective/absorption systems whose existence and properties is provided by
Theorems \ref{T1} and \ref{T2}.

\section{Sub-supersolutions theorems}

\label{S2}

In the sequel, the Banach spaces $W^{1,p_{i}}(\Omega )$ and $%
L^{p_{i}}(\Omega )$ are equipped with the usual norms $\Vert \cdot \Vert
_{1,p_{i}}$ and $\Vert \cdot \Vert _{p_{i}}$, respectively, whereas the
space $W_{0}^{1,p}(\Omega )$ is endowed with the equivalent norm $\Vert
\nabla u\Vert _{p}=(\int_{\Omega }|\nabla u|^{p}\,\mathrm{d}x)^{\frac{1}{p}}$%
. We denote by $p_{i}^{\ast }=\frac{Np_{i}}{N-p_{i}}$ and $p_{i}^{\prime }=%
\frac{p_{i}}{p_{i}-1}.$ The continuity of embedding $W_{0}^{1,p_{i}}(\Omega
) $ in $L^{r_{i}}(\Omega ),$ for $1\leq r_{i}\leq p_{i}^{\ast },$ guarantees
the existence of a constant $c_{r}>0$ such that 
\begin{equation}
\Vert u\Vert _{r_{i}}\leq c_{r_{i}}\left\Vert u\right\Vert _{1,p_{i}}\ \text{
for all}\ u\in W_{0}^{1,p_{i}}(\Omega ).  \label{embed}
\end{equation}%
We also utilize the Banach space $\mathcal{C}^{1}(\overline{\Omega })$ and $%
\mathcal{C}_{0}^{1,\sigma }(\overline{\Omega })$ such that for $\sigma \in
(0,1)$ we have%
\begin{equation*}
\mathcal{C}_{0}^{1,\sigma }(\overline{\Omega })=\left\{ u\in \mathcal{C}%
^{1,\sigma }(\overline{\Omega }):u=0\text{ \ on }\partial \Omega \right\} .
\end{equation*}%
We recall that a sub-supersolution for $(\mathrm{P})$ consists of two pairs $%
(\underline{u}_{1}$,$\underline{u}_{2}),(\overline{u}_{1},\overline{u}%
_{2})\in (W_{0}^{1,p_{1}}(\Omega )\cap L^{\infty }(\Omega ))\times
(W_{0}^{1,p_{2}}(\Omega )\cap L^{\infty }(\Omega ))$ such that there hold $%
\overline{u}_{i}\geq \underline{u}_{i}$ in $\Omega $, and%
\begin{equation}
\begin{array}{l}
\int_{\Omega }\left\vert \nabla \underline{u}_{1}\right\vert
^{p_{1}-2}\nabla \underline{u}_{1}\nabla \varphi _{1}\ \mathrm{d}%
x-\int_{\Omega }f_{1}(x,\underline{u}_{1},w_{2},\nabla \underline{u}%
_{1},\nabla w_{2})\varphi _{1}\ \mathrm{d}x\leq 0, \\ 
\int_{\Omega }\left\vert \nabla \underline{u}_{2}\right\vert
^{p_{2}-2}\nabla \underline{u}_{2}\nabla \varphi _{2}\ \mathrm{d}%
x-\int_{\Omega }f_{2}(x,w_{1},\underline{u}_{2},\nabla w_{1},\nabla 
\underline{u}_{2})\varphi _{2}\ \mathrm{d}x\leq 0,%
\end{array}
\label{sub}
\end{equation}%
\begin{equation}
\begin{array}{l}
\int_{\Omega }\left\vert \nabla \overline{u}_{1}\right\vert ^{p_{1}-2}\nabla 
\overline{u}_{1}\nabla \varphi _{1}\ \mathrm{d}x-\int_{\Omega }f_{1}(x,%
\overline{u}_{1},w_{2},\nabla \overline{u}_{1},\nabla w_{2})\varphi _{1}\ 
\mathrm{d}x\geq 0, \\ 
\int_{\Omega }\left\vert \nabla \overline{u}_{2}\right\vert ^{p_{2}-2}\nabla 
\overline{u}_{2}\nabla \varphi _{2}\ \mathrm{d}x-\int_{\Omega }f_{2}(x,w_{1},%
\overline{u}_{2},\nabla w_{1},\nabla \overline{u}_{2})\varphi _{2}\ \mathrm{d%
}x\geq 0,%
\end{array}
\label{sup}
\end{equation}%
for all $\varphi _{i}\in W_{0}^{1,p_{i}}\left( \Omega \right) $ with $%
\varphi _{i}\geq 0$ a.e. in $\Omega $ and for all $w_{i}\in
W_{0}^{1,p_{i}}\left( \Omega \right) $ satisfying $w_{i}\in \lbrack 
\underline{u}_{i},\overline{u}_{i}]$ a.e. in $\Omega $, for $i=1,2$.

For an arbitrary constant $\Lambda >0$, we introduce the bounded convex sets 
\begin{equation}
\mathcal{K}_{\Lambda }^{p}=\left\{ (y_{1},y_{2})\in W_{0}^{1,p_{1}}(\Omega
)\times W_{0}^{1,p_{2}}(\Omega ):\Vert \nabla y_{1}\Vert _{p_{1}},\Vert
\nabla y_{2}\Vert _{p_{2}}<\Lambda \right\} ,  \label{18}
\end{equation}%
\begin{equation}
\mathcal{K}_{\Lambda }^{\infty }=\left\{ (y_{1},y_{2})\in \mathcal{C}^{1}(%
\overline{\Omega })\times \mathcal{C}^{1}(\overline{\Omega }):\Vert \nabla
y_{1}\Vert _{\infty },\Vert \nabla y_{2}\Vert _{\infty }<\Lambda \right\} .
\label{18*}
\end{equation}%
Corresponding to sub-supersolutions pairs $(\underline{u}_{1},\underline{u}%
_{2}),(\overline{u}_{1},\overline{u}_{2})$ and for $(z_{1},z_{2})\in 
\mathcal{K}_{\Lambda }^{p}$ (resp. $(z_{1},z_{2})\in \mathcal{K}_{\Lambda
}^{\infty }$), we introduce the auxiliary system%
\begin{equation*}
(\mathrm{P}_{z})\qquad \left\{ 
\begin{array}{l}
-\Delta _{p_{i}}u_{i}=\tilde{f}_{i}(x,z_{1},z_{2},\nabla z_{1},\nabla z_{2})%
\text{ in }\Omega \\ 
u_{i}=0\text{ \ \ \ on }\partial \Omega ,i=1,2,%
\end{array}%
\right.
\end{equation*}%
where 
\begin{equation}
\tilde{f}_{i}(x,z_{1},z_{2},\nabla z_{1},\nabla z_{2})=f_{i}(x,\tilde{z}_{1},%
\tilde{z}_{2},\nabla \tilde{z}_{1},\nabla \tilde{z}_{2}),  \label{25}
\end{equation}%
\begin{equation}
\tilde{z}_{i}=\min (\max (z_{i},\underline{u}_{i}),\overline{u}_{i})\text{, }%
i=1,2,  \label{26}
\end{equation}%
to which we associate the solution operators $\mathcal{T}_{p}$ and $\mathcal{%
T}_{\infty },$ respectively, defined as follows:%
\begin{equation*}
\begin{array}{cccc}
\mathcal{T}_{p}: & \mathcal{K}_{\Lambda }^{p} & \rightarrow & 
W_{0}^{1,p_{1}}(\Omega )\times W_{0}^{1,p_{2}}(\Omega ) \\ 
& (z_{1},z_{2}) & \mapsto & (u_{1},u_{2}),%
\end{array}%
\end{equation*}%
\begin{equation*}
\begin{array}{cccc}
\mathcal{T}_{\infty }: & \mathcal{K}_{\Lambda }^{\infty } & \rightarrow & 
\mathcal{C}_{0}^{1}(\overline{\Omega })\times \mathcal{C}_{0}^{1}(\overline{%
\Omega }) \\ 
& (z_{1},z_{2}) & \mapsto & (u_{1},u_{2}).%
\end{array}%
\end{equation*}%
It should be noted that, subsequent to the implementation of specific growth
conditions on nonlinearities $\tilde{f}_{1}$ and $\tilde{f}_{2}$, it will be
shown that the operators $\mathcal{T}_{p}$ and $\mathcal{T}_{\infty }$ are
well defined.

We assume the following assumption

\begin{description}
\item[$(\mathrm{H})$] There exist constants $M_{i}>0,$ $\mu _{i}\in 
\mathbb{R}
$ and $0\leq \gamma _{i},\theta _{i}<p_{i}-1$ such that%
\begin{equation*}
\left\vert f_{i}(x,s_{1},s_{2},\xi _{1},\xi _{2})\right\vert \leq
M_{i}(d(x)^{\mu _{i}}+\left\vert \xi _{1}\right\vert ^{\gamma
_{i}}+\left\vert \xi _{2}\right\vert ^{\theta _{i}}),
\end{equation*}%
for a.e. $x\in \Omega ,$ for all $s_{i}\in \left[ \underline{u}_{i},%
\overline{u}_{i}\right] ,$ and all $\xi _{i}\in 
\mathbb{R}
^{N},$ for $i=1,2.$
\end{description}

The existence result for problem $(\mathrm{P})$ is formulated as follows.

\begin{theorem}
\label{T1}Let $(\underline{u}_{1}$,$\underline{u}_{2}),(\overline{u}_{1},%
\overline{u}_{2})$ be sub-supersolution pairs of $(\mathrm{P})$ and assume $(%
\mathrm{H})$ fulfilled with%
\begin{equation}
\gamma _{1}\leq \frac{p_{1}}{(p_{1}^{\ast })^{\prime }},\text{ }\theta
_{2}\leq \frac{p_{2}}{(p_{2}^{\ast })^{\prime }},\text{ }\gamma _{2}\leq 
\frac{p_{1}}{p_{2}^{\prime }},\text{ \ }\theta _{1}\leq \frac{p_{2}}{%
p_{1}^{\prime }}  \label{5*}
\end{equation}%
and%
\begin{equation}
\mu _{i}>-1+\frac{1}{p_{i}}.  \label{6*}
\end{equation}%
Then, system $(\mathrm{P})$ has a solution $(u_{1},u_{2})\in
W_{0}^{1,p_{1}}(\Omega )\times W_{0}^{1,p_{2}}(\Omega )$ within $\left[ 
\underline{u}_{1},\overline{u}_{1}\right] \times \left[ \underline{u}_{2},%
\overline{u}_{2}\right] .$
\end{theorem}

\begin{remark}
\label{R3}Assumption (\ref{5*}) guarantees that the continuous embedding $%
L^{p_{1}^{\ast }}(\Omega )\hookrightarrow L^{\frac{p_{1}}{p_{1}-\gamma _{1}}%
}(\Omega )$ holds true because%
\begin{equation*}
\frac{p_{1}-\gamma _{1}}{p_{1}}=1-\frac{\gamma _{1}}{p_{1}}>1-\frac{1}{%
(p_{1}^{\ast })^{\prime }}=\frac{1}{p_{1}^{\ast }},
\end{equation*}%
so $\frac{p_{1}}{p_{1}-\gamma _{1}}<p_{1}^{\ast }.$ Similarly, we get the
continuous embedding $L^{p_{2}^{\ast }}(\Omega )\hookrightarrow L^{\frac{%
p_{2}}{p_{2}-\theta _{2}}}(\Omega ).$
\end{remark}

\begin{proof}[Proof of Theorem \protect\ref{T1}]
First, we claim that there exists a constant $R>0$ such that all solutions $%
(u_{1},u_{2})\in W_{0}^{1,p_{1}}(\Omega )\times W_{0}^{1,p_{2}}(\Omega )$ of 
$(\mathrm{P})$ satisfy the estimate%
\begin{equation}
\left\Vert \nabla u_{1}\right\Vert _{p_{1}},\left\Vert \nabla
u_{2}\right\Vert _{p_{2}}<R.  \label{R}
\end{equation}%
Let $(u_{1},u_{2})\in W_{0}^{1,p_{1}}(\Omega )\times W_{0}^{1,p_{2}}(\Omega
) $ be a solution of $(\mathrm{P})$. Multiplying the first and the second
equation in $(\mathrm{P})$ by $u_{1}$ and $u_{2},$\ respectively,
integrating over $\Omega $ and using $(\mathrm{H})$, we get%
\begin{equation}
\begin{array}{l}
\int_{\Omega }|\nabla u_{i}|^{p_{i}}\text{ }\mathrm{d}x=\int_{\Omega
}f_{i}(x,u_{1},u_{2},\nabla u_{1},\nabla u_{2})u_{i}\text{ }\mathrm{d}x \\ 
\leq M_{i}\int_{\Omega }(d(x)^{\mu _{i}}+\left\vert \nabla u_{1}\right\vert
^{\gamma _{i}}+\left\vert \nabla u_{2}\right\vert ^{\theta _{i}})u_{i}\text{ 
}\mathrm{d}x.%
\end{array}
\label{43}
\end{equation}%
If $\mu _{i}\geq 0$, it readily seen that%
\begin{equation}
\int_{\Omega }d(x)^{\mu _{i}}u_{i}\text{ }\mathrm{d}x\leq c||d(x)||_{\infty
}^{\mu _{i}}\left\Vert \nabla u_{i}\right\Vert _{p_{i}},  \label{46*}
\end{equation}%
for a certain constant $c>0.$ For $\mu _{i}<0$ in (\ref{6*}), by
Hardy-Sobolev inequality (see, e.g., \cite[Lemma 2.3]{AC}), there exists a
positive constant $C_{i}$ such that%
\begin{equation}
\int_{\Omega }d(x)^{\mu _{i}}u_{i}\text{ }\mathrm{d}x\leq C_{i}\Vert \nabla
u_{i}\Vert _{p_{i}}\text{.}  \label{46}
\end{equation}%
Using (\ref{embed}), (\ref{5*}), Young inequality and Remark \ref{R3}, it
follows that%
\begin{equation}
\begin{array}{l}
\int_{\Omega }\left\vert \nabla u_{1}\right\vert ^{\gamma _{1}}u_{1}\text{ }%
\mathrm{d}x\leq \epsilon \int_{\Omega }|\nabla u_{1}|^{p_{1}}\text{ }\mathrm{%
d}x+c_{\epsilon }\int_{\Omega }|u_{1}|^{\frac{p_{1}}{p_{1}-\gamma _{1}}}%
\text{ }\mathrm{d}x \\ 
\leq \epsilon \left\Vert \nabla u_{1}\right\Vert
_{p_{1}}^{p_{1}}+c_{\epsilon }c_{0}\left\Vert \nabla u_{1}\right\Vert
_{p_{1}}^{\frac{p_{1}}{p_{1}-\gamma _{1}}},%
\end{array}%
\end{equation}%
\begin{equation}
\begin{array}{l}
\int_{\Omega }\left\vert \nabla u_{2}\right\vert ^{\theta _{1}}u_{1}\text{ }%
\mathrm{d}x\leq \epsilon \int_{\Omega }|u_{1}|^{p_{1}}\text{ }\mathrm{d}%
x+c_{\epsilon }\int_{\Omega }|\nabla u_{2}|^{\theta _{1}p_{1}^{\prime }}%
\text{ }\mathrm{d}x \\ 
\leq \epsilon c_{1}\left\Vert \nabla u_{1}\right\Vert
_{p_{1}}^{p_{1}}+c_{\epsilon }\left\Vert \nabla u_{2}\right\Vert
_{p_{2}}^{\theta _{1}p_{1}^{\prime }},%
\end{array}%
\end{equation}%
for constants $c_{0},c_{1}>0$, for every $\epsilon >0$ and for a certain
constant $c_{\epsilon }>0$ depending on $\epsilon $.

Repeating the previous argument leads to%
\begin{equation}
\begin{array}{l}
\int_{\Omega }\left\vert \nabla u_{2}\right\vert ^{\theta _{2}}u_{2}\text{ }%
\mathrm{d}x\leq \epsilon \int_{\Omega }|\nabla u_{2}|^{p_{2}}\text{ }\mathrm{%
d}x+c_{\epsilon }\int_{\Omega }|u_{2}|^{\frac{p_{2}}{p_{2}-\theta _{2}}}%
\text{ }\mathrm{d}x \\ 
\leq \epsilon \left\Vert \nabla u_{2}\right\Vert
_{p_{2}}^{p_{2}}+c_{\epsilon }^{\prime }c_{0}^{\prime }\left\Vert \nabla
u_{2}\right\Vert _{p_{2}}^{\frac{p_{2}}{p_{2}-\theta _{2}}},%
\end{array}%
\end{equation}%
\begin{equation}
\begin{array}{l}
\int_{\Omega }\left\vert \nabla u_{1}\right\vert ^{\gamma _{2}}u_{2}\text{ }%
\mathrm{d}x\leq \epsilon \int_{\Omega }|u_{2}|^{p_{2}}\text{ }\mathrm{d}%
x+c_{\epsilon }\int_{\Omega }|\nabla u_{1}|^{\gamma _{2}p_{2}^{\prime }}%
\text{ }\mathrm{d}x \\ 
\leq \epsilon c_{1}^{\prime }\left\Vert \nabla u_{2}\right\Vert
_{p_{2}}^{p_{2}}+c_{\epsilon }^{\prime }\left\Vert \nabla u_{1}\right\Vert
_{p_{1}}^{\gamma _{2}p_{2}^{\prime }},%
\end{array}
\label{44}
\end{equation}%
for $\epsilon >0$ and for constants $c_{\epsilon }^{\prime },c_{0}^{\prime
},c_{1}^{\prime }>0$.

Then, on account of\ (\ref{5*}) and the fact that $p_{i}>1,$ gathering (\ref%
{43})-(\ref{44}) together and taking $\varepsilon $ small, there is a
constant $R>0,$ independent of $(u_{1},u_{2})$, such that (\ref{R}) holds
true. This proves the claim.

On the basis of the a priori estimate $R>0$, we consider the set $\mathcal{K}%
_{R}^{p}$ defined in (\ref{18}). For $(z_{1},z_{2})\in \mathcal{K}_{R}^{p}$
and by $(\mathrm{H}),$ (\ref{5*}), using H\"{o}lder's inequality together
with (\ref{embed}) and Remark \ref{R3}, we get%
\begin{equation}
\begin{array}{l}
\int\limits_{\Omega }\tilde{f}_{i}(x,z_{1},z_{2},\nabla z_{1},\nabla z_{2})%
\text{ }\mathrm{d}x\leq \int\limits_{\Omega }M_{i}(d(x)^{\mu
_{i}}+\left\vert \nabla \tilde{z}_{1}\right\vert ^{\gamma _{i}}+\left\vert
\nabla \tilde{z}_{2}\right\vert ^{\theta _{i}})\text{ }\mathrm{d}x \\ 
\leq M_{i}(\int\limits_{\Omega }d(x)^{\mu _{i}}\text{ }\mathrm{d}x+|\Omega
|^{\frac{p_{1}-\gamma _{i}}{p_{1}}}\left\Vert \nabla \tilde{z}%
_{1}\right\Vert _{p_{1}}^{^{\gamma _{i}}}+|\Omega |^{\frac{p_{2}-\theta _{i}%
}{p_{2}}}\left\Vert \nabla \tilde{z}_{2}\right\Vert _{p_{2}}^{\theta _{i}})
\\ 
\leq M_{i}\int\limits_{\Omega }d(x)^{\mu _{i}}\text{ }\mathrm{d}%
x+c_{i}(R^{^{\gamma _{i}}}+R^{\theta _{i}}).%
\end{array}
\label{27}
\end{equation}%
with certain constant $c_{i}>0,i=1,2.$ If $\mu _{i}\geq 0$, it is readily
seen that $\int\limits_{\Omega }d(x)^{\mu _{i}}$ $\mathrm{d}x<\infty $
whereas, if $\mu _{i}<0$, the same conclusion follows on the basis of (\ref%
{6*}) and \cite[Lemma in page 726]{LM}. This shows that $\tilde{f}_{i}\in
W^{-1,p_{i}^{^{\prime }}}(\Omega ).$ Hence, the unique solvability of $%
(u_{1},u_{2})$ in $(\mathrm{P}_{z})$ is readily derived from Minty-Browder
Theorem (see, e.g., \cite[Theorem V.15]{Brezis}). This ensures that the
operator $\mathcal{T}_{p}:\mathcal{K}_{R}^{p}\rightarrow
W_{0}^{1,p_{1}}(\Omega )\times W_{0}^{1,p_{2}}(\Omega )$ is well defined.

We note from $(\mathrm{P}_{z})$ that any fixed point of $\mathcal{T}_{p}$
coincides with the weak solution of problem $(\mathrm{P}).$ Consequently, to
achieve the desired conclusion it suffices to prove that $\mathcal{T}_{p}$
has a fixed point. To this end we apply Schauder's fixed point theorem.

We show that $\mathcal{T}_{p}$ is continuous. Let $(z_{1,n},z_{2,n})%
\rightarrow (z_{1},z_{2})$ in $\mathcal{K}_{R}^{p},$ for all $n$. Denote $%
(u_{1,n},u_{2,n})=\mathcal{T}_{p}(z_{1,n},z_{2,n}),$ which reads as%
\begin{equation}
\int\limits_{\Omega }\left\vert \nabla u_{i,n}\right\vert ^{p_{i}-2}\nabla
u_{i,n}\nabla \varphi _{i}dx=\int\limits_{\Omega }\tilde{f}%
_{i}(x,z_{1,n},z_{2,n},\nabla z_{1,n},\nabla z_{2,n})\varphi _{i}\text{ }%
\mathrm{d}x,  \label{28}
\end{equation}%
for all $\varphi _{i}\in W_{0}^{1,p_{i}}(\Omega )$. Acting in (\ref{28})
with $\varphi _{i}=u_{i,n}$, using $(\mathrm{H}),$ we get%
\begin{equation}
\begin{array}{l}
\left\Vert \nabla u_{1,n}\right\Vert _{p_{1}}^{p_{1}}=\int\limits_{\Omega }%
\tilde{f}_{1}(x,z_{1,n},z_{2,n},\nabla z_{1,n},\nabla z_{2,n})u_{1,n}\text{ }%
\mathrm{d}x \\ 
\leq M_{1}\int\limits_{\Omega }(d(x)^{\mu _{1}}+\left\vert \nabla
z_{1,n}\right\vert ^{\gamma _{1}}+\left\vert \nabla z_{2,n}\right\vert
^{\theta _{1}})u_{1,n}\text{ }\mathrm{d}x%
\end{array}
\label{29}
\end{equation}%
and%
\begin{equation}
\begin{array}{l}
\left\Vert \nabla u_{2,n}\right\Vert _{p_{2}}^{p_{2}}=\int\limits_{\Omega }%
\tilde{f}_{2}(x,z_{1,n},z_{2,n},\nabla z_{1,n},\nabla z_{2,n})u_{2,n}\text{ }%
\mathrm{d}x \\ 
\leq M_{2}\int\limits_{\Omega }(d(x)^{\mu _{2}}+\left\vert \nabla
z_{1,n}\right\vert ^{\gamma _{2}}+\left\vert \nabla z_{2,n}\right\vert
^{\theta _{2}})u_{2,n}\text{ }\mathrm{d}x.%
\end{array}
\label{29*}
\end{equation}%
From (\ref{46*}) or (\ref{46}) depending on whether $\mu _{i}<0$ or $\mu
_{i}\geq 0,$ bearing in mind that $\frac{p_{1}}{p_{1}-\gamma _{1}}%
<p_{1}^{\ast }$ and $\theta _{1}p_{1}^{\prime }\leq p_{2}$, H\"{o}lder's
inequality together with (\ref{5*}) and the embedding (\ref{embed}) imply%
\begin{equation}
\begin{array}{l}
\int\limits_{\Omega }(d(x)^{\mu _{1}}+\left\vert \nabla z_{1,n}\right\vert
^{\gamma _{1}}+\left\vert \nabla z_{2,n}\right\vert ^{\theta _{1}})u_{1,n}%
\text{ }\mathrm{d}x \\ 
\leq C_{0}\left\Vert \nabla u_{1,n}\right\Vert _{p_{1}}+\left\Vert \nabla
z_{1,n}\right\Vert _{p_{1}}^{\gamma _{1}}\left\Vert u_{1,n}\right\Vert _{%
\frac{p_{1}}{p_{1}-\gamma _{1}}}+\left\Vert \nabla z_{2,n}\right\Vert
_{\theta _{1}p_{1}^{\prime }}^{\theta _{1}}\left\Vert u_{1,n}\right\Vert
_{p_{1}} \\ 
\leq C_{1}\left\Vert \nabla u_{1,n}\right\Vert _{p_{1}}(1+\left\Vert \nabla
z_{1,n}\right\Vert _{p_{1}}^{\gamma _{1}}+\left\Vert \nabla
z_{2,n}\right\Vert _{p_{2}}^{\theta _{1}}).%
\end{array}
\label{30}
\end{equation}%
Similarly, recalling that $\frac{p_{2}}{p_{2}-\theta _{2}}<p_{2}^{\ast }$
and $\gamma _{2}p_{2}^{\prime }\leq p_{1}$, we get%
\begin{equation}
\begin{array}{l}
\int\limits_{\Omega }(d(x)^{\mu _{2}}+\left\vert \nabla z_{1,n}\right\vert
^{\gamma _{2}}+\left\vert \nabla z_{2,n}\right\vert ^{\theta _{2}})u_{2,n}%
\text{ }\mathrm{d}x \\ 
\leq C_{0}^{\prime }\left\Vert \nabla u_{2,n}\right\Vert _{p_{2}}+\left\Vert
\nabla z_{1,n}\right\Vert _{\gamma _{2}p_{2}^{\prime }}^{\gamma
_{2}}\left\Vert u_{2,n}\right\Vert _{p_{2}}+\left\Vert \nabla
z_{2,n}\right\Vert _{p_{2}}^{\theta _{2}}\left\Vert u_{2,n}\right\Vert _{%
\frac{p_{2}}{p_{2}-\gamma _{2}}} \\ 
\leq C_{2}\left\Vert \nabla u_{2,n}\right\Vert _{p_{2}}(1+\left\Vert \nabla
z_{1,n}\right\Vert _{p_{1}}^{\gamma _{2}}+\left\Vert \nabla
z_{2,n}\right\Vert _{p_{2}}^{\theta _{2}}),%
\end{array}
\label{30*}
\end{equation}%
with constants $C_{1},C_{2}>0.$

Gathering (\ref{29})-(\ref{30*}) together, since $p_{i}>1$ and $%
(z_{1},z_{2})\in \mathcal{K}_{R}^{p},$ it follows that $\left\{
u_{i,n}\right\} $ is bounded in $W_{0}^{1,p_{i}}(\Omega )$ for $i=1,2$. So,
passing to relabeled subsequences, we can write the weak convergence%
\begin{equation}
u_{i,n}\rightharpoonup u_{i},\text{ for some }u_{i}\in
W_{0}^{1,p_{i}}(\Omega ).  \label{32}
\end{equation}%
Setting $\varphi _{i}=u_{i,n}-u_{i}$ in (\ref{28}) we find that%
\begin{equation*}
\int\limits_{\Omega }\left\vert \nabla u_{i,n}\right\vert ^{p_{i}-2}\nabla
u_{i,n}\nabla (u_{i,n}-u_{i})\text{ }\mathrm{d}x=\int\limits_{\Omega }\tilde{%
f}_{i}(x,z_{1,n},z_{2,n},\nabla z_{1,n},\nabla z_{2,n})(u_{i,n}-u_{i})\text{ 
}\mathrm{d}x.
\end{equation*}%
By $(\mathrm{H})$ we have%
\begin{equation}
\begin{array}{l}
\int\limits_{\Omega }\tilde{f}_{i}(x,z_{1,n},z_{2,n},\nabla z_{1,n},\nabla
z_{2,n})(u_{i,n}-u_{i})\text{ }\mathrm{d}x \\ 
\leq M_{i}\int\limits_{\Omega }(d(x)^{\mu _{i}}+\left\vert \nabla
z_{1,n}\right\vert ^{\gamma _{i}}+\left\vert \nabla z_{2,n}\right\vert
^{\theta _{i}})(u_{i,n}-u_{i})\text{ }\mathrm{d}x.%
\end{array}
\label{33}
\end{equation}%
We point out that (\ref{32}) together with (\ref{30}) and (\ref{30*}), where 
$u_{i,n}$ is replaced by $u_{i,n}-u_{i},$ enable us to derive that%
\begin{equation}
\int\limits_{\Omega }\left\vert \nabla z_{1,n}\right\vert ^{\gamma
_{1}}(u_{1,n}-u_{1})\text{ }\mathrm{d}x\leq \left\Vert \nabla
z_{1,n}\right\Vert _{p_{1}}^{\gamma _{1}}\left\Vert u_{1,n}-u_{1}\right\Vert
_{\frac{p_{1}}{p_{1}-\gamma _{1}}}\rightarrow 0,
\end{equation}%
\begin{equation}
\int\limits_{\Omega }\left\vert \nabla z_{2,n}\right\vert ^{\theta
_{1}}(u_{1,n}-u_{1})\text{ }\mathrm{d}x\leq \left\Vert \nabla
z_{2,n}\right\Vert _{\theta _{1}p_{1}^{\prime }}^{\theta _{1}}\left\Vert
u_{1,n}-u_{1}\right\Vert _{p_{1}}\rightarrow 0,
\end{equation}%
\begin{equation}
\int\limits_{\Omega }\left\vert \nabla z_{2,n}\right\vert ^{\theta
_{2}}(u_{2,n}-u_{2})\text{ }\mathrm{d}x\leq \left\Vert \nabla
z_{2,n}\right\Vert _{p_{2}}^{\theta _{2}}\left\Vert u_{2,n}-u_{2}\right\Vert
_{\frac{p_{2}}{p_{2}-\theta _{2}}}\rightarrow 0,
\end{equation}%
\begin{equation}
\int\limits_{\Omega }\left\vert \nabla z_{1,n}\right\vert ^{\gamma
_{2}}(u_{2,n}-u_{2})\text{ }\mathrm{d}x\leq \left\Vert \nabla
z_{1,n}\right\Vert _{\gamma _{2}p_{2}^{\prime }}^{\gamma _{2}}\left\Vert
u_{2,n}-u_{2}\right\Vert _{p_{2}}\rightarrow 0.
\end{equation}%
Moreover, for $\mu _{i}\geq 0$, by Holder's inequality and (\ref{32}), we
have%
\begin{equation}
\int\limits_{\Omega }d(x)^{\mu _{i}}(u_{i,n}-u_{i})\text{ }\mathrm{d}x\leq
||d(x)||_{\infty }^{\mu _{i}}\left\Vert u_{1,n}-u_{1}\right\Vert
_{1}\rightarrow 0,
\end{equation}%
while, if $\mu _{i}<0,$ by means of \cite[Lemma in page 726]{LM} together
with (\ref{6*}), we get%
\begin{equation}
\int\limits_{\Omega }d(x)^{\mu _{i}}(u_{i,n}-u_{i})\text{ }\mathrm{d}x\leq
||d(x)||_{\frac{\mu _{i}p_{i}}{p_{i}-1}}^{\mu _{i}}\Vert u_{i,n}-u_{i}\Vert
_{p_{i}}\rightarrow 0.  \label{33*}
\end{equation}%
Here, note that from (\ref{6*}) implies $\frac{\mu _{i}p_{i}}{p_{i}-1}>-1$
and therefore, by \cite[Lemma in page 726]{LM}, $d(x)^{\frac{\mu _{i}p_{i}}{%
p_{i}-1}}\in L^{1}(\Omega ).$

Combining (\ref{33})-(\ref{33*}) results in%
\begin{equation*}
\lim_{n\rightarrow +\infty }\int\limits_{\Omega }\tilde{f}%
_{i}(x,z_{1,n},z_{2,n},\nabla z_{1,n},\nabla z_{2,n})(u_{i,n}-u_{i})\text{ }%
\mathrm{d}x=0,\text{ for }i=1,2.
\end{equation*}%
Then, we see that $\lim_{n\longrightarrow \infty }\left\langle -\Delta
_{p_{i}}u_{i,n},u_{i,n}-u_{i}\right\rangle =0,$ $i=1,2.$ The $S_{+}$
-property of $-\Delta _{p_{i}}$ on $W_{0}^{1,p_{i}}(\Omega )$ along with (%
\ref{32}) implies%
\begin{equation}
u_{i,n}\rightarrow u_{i}\text{ in }W_{0}^{1,p_{i}}(\Omega )\text{ }i=1,2.
\label{34}
\end{equation}%
Thus, passing to the limit in (\ref{28}) leads to $(u_{1},u_{2})=\mathcal{T}%
_{p}(z_{1},z_{2})$. This proves the claim.

Following the same reasoning as before, by considering a bounded sequence $%
(u_{1,n},u_{2,n})$ in $W_{0}^{1,p_{1}}(\Omega )\times W_{0}^{1,p_{2}}(\Omega
)$ with $(u_{1,n},u_{2,n})=\mathcal{T}_{p}(z_{1,n},z_{2,n}),$ for $%
(z_{1,n},z_{2,n})\in \mathcal{K}_{R}^{p}$, we find $(u_{1},u_{2})\in
W_{0}^{1,p_{1}}(\Omega )\times W_{0}^{1,p_{2}}(\Omega )$ such that, along a
relabeled subsequence, $(u_{1,n},u_{2,n})\rightarrow (u_{1},u_{2})$ in $%
W_{0}^{1,p_{1}}(\Omega )\times W_{0}^{1,p_{2}}(\Omega )$, thereby the
relative compactness of $\mathcal{T}_{p}(\mathcal{K}_{R}^{p})$ is proven.

Now, by considering $(\mathrm{P}_{z})$ and arguing as in (\ref{29})-(\ref%
{30*}) with $(z_{1},z_{2})$ in place of $(z_{1,n},z_{2,n})$, due to $(%
\mathrm{H})$ and for $R>0$ sufficiently large, we get%
\begin{equation*}
\left\Vert \nabla u_{i}\right\Vert _{p_{i}}^{p_{i}-1}\leq c_{i}(1+\left\Vert
\nabla z_{1}\right\Vert _{p_{1}}^{\gamma _{i}}+\left\Vert \nabla
z_{2}\right\Vert _{p_{2}}^{\theta _{i}})\leq c_{i}(1+R^{\gamma
_{i}}+R^{\theta _{i}})\leq R^{p_{i}-1},
\end{equation*}%
showing that $\mathcal{K}_{R}^{p}$ is invariant under $\mathcal{T}_{p}$,
that is, $\mathcal{T}_{p}(\mathcal{K}_{R}^{p})\subset \mathcal{K}_{R}^{p}.$

We are thus in a position to apply Schauder's fixed point theorem to the map 
$\mathcal{T}_{p}:\mathcal{K}_{R}^{p}\rightarrow \mathcal{K}_{R}^{p}$ which
establishes the existence of $(u_{1},u_{2})\in \mathcal{K}_{R}^{p}$
satisfying $(u_{1},u_{2})=\mathcal{T}(u_{1},u_{2})$.

Let us justify that%
\begin{equation*}
\underline{u}_{1}\leq u_{1}\leq \overline{u}_{1}\text{ and }\underline{u}%
_{2}\leq u_{2}\leq \overline{u}_{2}\text{ in }\Omega .
\end{equation*}%
Set $\zeta _{1}=(\underline{u}_{1}-u_{1})^{+}$ and suppose $\zeta _{1}\neq 0$%
. Then, from $(\mathrm{P}_{z})$, (\ref{25}), (\ref{26}) and (\ref{sub}), we
infer that%
\begin{equation*}
\begin{array}{l}
\int_{\{\underline{u}_{1}>u_{1}\}}|\nabla u_{1}|^{p_{1}-2}\nabla u_{1}\nabla
\zeta _{1}\ \mathrm{d}x=\int_{\Omega }|\nabla u_{1}|^{p_{1}-2}\nabla
u_{1}\nabla \zeta _{1}\ \mathrm{d}x \\ 
=\int_{\{\underline{u}_{1}>u_{1}\}}\tilde{f}_{1}(x,u_{1},u_{2},\nabla
u_{1},\nabla u_{2})\zeta _{1}\ \mathrm{d}x=\int_{\{\underline{u}%
_{1}>u_{1}\}}f_{1}(x,\tilde{u}_{1},\tilde{u}_{2},\nabla \tilde{u}_{1},\nabla 
\tilde{u}_{2})\zeta _{1}\ \mathrm{d}x \\ 
=\int_{\{\underline{u}_{1}>u_{1}\}}f_{1}(x,\underline{u}_{1},\tilde{u}%
_{2},\nabla \underline{u}_{1},\nabla \tilde{u}_{2})\zeta _{1}\ \mathrm{d}%
x\geq \int_{\{\underline{u}_{1}>u_{1}\}}|\nabla \underline{u}%
_{1}|^{p_{1}-2}\nabla \underline{u}_{1}\nabla \zeta _{1}\ \mathrm{d}x.%
\end{array}%
\end{equation*}%
This implies that%
\begin{equation*}
\begin{array}{c}
\int_{\{\underline{u}_{1}>u_{1}\}}(|\nabla u_{1}|^{p_{1}-2}\nabla
u_{1}-|\nabla \underline{u}_{1}|^{p_{1}-2}\nabla \underline{u}_{1})\nabla
\zeta _{1}\ \mathrm{d}x\geq 0,%
\end{array}%
\end{equation*}%
a contradiction. Hence $u_{1}\geq \underline{u}_{1}$ in $\Omega $. A quite
similar argument provides that $u_{2}\geq \underline{u}_{2}$ in $\Omega $.
In the same way, we prove that $u_{1}\leq \overline{u}_{1}$ and $u_{2}\leq 
\overline{u}_{2}$ in $\Omega $. This completes the proof.
\end{proof}

Next we set forth a result addressing the regularity of solutions and a
priori estimates for problem $(\mathrm{P})$. Under more restrictive
assumption on the exponents in $(\mathrm{H}),$ regular solutions for $(%
\mathrm{P})$ are obtained.

\begin{theorem}
\label{T2}Let $(\underline{u}_{1}$,$\underline{u}_{2}),(\overline{u}_{1},%
\overline{u}_{2})$ be sub-supersolution pairs of $(\mathrm{P})$ and assume $(%
\mathrm{H})$ fulfilled with%
\begin{equation}
\mu _{i}>-\frac{1}{r_{i}}\text{ \ and \ }\max \{\gamma _{i},\theta _{i}\}<%
\frac{p_{i}-1}{r_{i}},\text{ with }r_{i}>N,\text{ for }i=1,2.  \label{5}
\end{equation}%
Then, system $(\mathrm{P})$ has a solution $(u_{1},u_{2})\in \mathcal{C}%
_{0}^{1,\sigma }(\overline{\Omega })\times \mathcal{C}_{0}^{1,\sigma }(%
\overline{\Omega }),$ for certain $\sigma \in (0,1),$ within $\left[ 
\underline{u}_{1},\overline{u}_{1}\right] \times \left[ \underline{u}_{2},%
\overline{u}_{2}\right] .$
\end{theorem}

\begin{proof}
The proof is based on the same arguments as those developed to show Theorem %
\ref{T1}. Indeed, first observe that under assumptions $(\mathrm{H})$ and (%
\ref{5}), \cite[Theorem 2.1]{DM} guarantees that all solutions $%
(u_{1},u_{2}) $ of $(\mathrm{P})$ belong to $\mathcal{C}_{0}^{1,\sigma }(%
\overline{\Omega })\times \mathcal{C}_{0}^{1,\sigma }(\overline{\Omega }),$
for certain $\sigma \in (0,1),$ and there is a constant $L>0$ such that it
holds%
\begin{equation}
\left\Vert u_{1}\right\Vert _{C^{1,\sigma }(\overline{\Omega })},\left\Vert
u_{2}\right\Vert _{C^{1,\sigma }(\overline{\Omega })}<L\text{.}  \label{70}
\end{equation}%
Here, (\ref{70}) provides an a priori estimate to solutions of $(\mathrm{P})$%
, which is fundamental to setting up an argument based on Schauder's fixed
point Theorem. We employ the operator $\mathcal{T}_{\infty }$, defined on
the set $\mathcal{K}_{L}^{\infty }$ for the a priori estimate $L>0$, which
is associated to the auxiliary problem $(\mathrm{P}_{z})$.

For each $(z_{1},z_{2})\in \mathcal{K}_{L}^{\infty }$ and by $(\mathrm{H})$
with (\ref{5}), we have%
\begin{equation}
\begin{array}{l}
\tilde{f}_{i}(x,z_{1},z_{2},\nabla z_{1},\nabla z_{2}) \\ 
\leq M_{i}(d(x)^{\mu _{i}}+\left\vert \nabla \tilde{z}_{1}\right\vert
^{\gamma _{i}}+\left\vert \nabla \tilde{z}_{2}\right\vert ^{\theta _{i}}) \\ 
\leq M_{i}(d(x)^{\mu _{i}}+L^{\gamma _{i}}+L^{\theta _{i}}) \\ 
\leq M_{i}d(x)^{\mu _{i}}+2L^{\max \{\gamma _{i},\theta _{i}\}}\text{ in }%
\Omega ,\text{ for }i=1,2.%
\end{array}
\label{24}
\end{equation}%
Then, since $\mu _{i}>-1,$ Hardy-Sobolev inequality (see, e.g., \cite[Lemma
2.3]{AC}) guarantees that $\tilde{f}_{i}(x,z_{1},z_{2},\nabla z_{1},\nabla
z_{2})\in W^{-1,p_{i}^{\prime }}(\Omega )$, for $i=1,2$. Thus, the unique
solvability of $(u_{1},u_{2})$ in $(\mathrm{P}_{z})$ is readily derived from
Minty Browder's Theorem and therefore, $\mathcal{T}_{\infty }$ is well
defined. Moreover, the regularity theory up to the boundary in \cite{Hai}
yields $(u_{1},u_{2})\in \mathcal{C}_{0}^{1,\sigma }(\overline{\Omega }%
)\times \mathcal{C}_{0}^{1,\sigma }(\overline{\Omega }),$ $\sigma \in (0,1).$

The map $\mathcal{T}_{\infty }:\mathcal{K}_{L}^{\infty }\rightarrow
C_{0}^{1}(\overline{\Omega })\times C_{0}^{1}(\overline{\Omega })$ is
continuous and compact. The proof uses similar argument as the one developed
in the proof of Theorem \ref{T1}. However, it is important to note that the
integrability of the function $d(x)^{\mu _{i}}(u_{i,n}-u_{i}),$ which is
known by invoking the Hardy-Sobolev inequality under assumption $\mu _{i}>-1$
in (\ref{5}), is essential.

By (\ref{24}) and for $(z_{1},z_{2})\in \mathcal{K}_{R}^{\infty }$, \cite[%
Theorem 2.1]{DM} entails%
\begin{equation}
\begin{array}{l}
\Vert \nabla u_{i}\Vert _{\infty }\leq k_{p_{i}}||\tilde{f}_{i}(\cdot
,z_{1},z_{2},\nabla z_{1},\nabla z_{2})\Vert _{r_{i}}^{\frac{1}{p_{i}-1}} \\ 
\leq k_{0}\left( \left( \int_{\Omega }d(x)^{r_{i}\mu _{i}}\text{ }\mathrm{d}%
x\right) ^{\frac{1}{r_{i}(p_{i}-1)}}+|\Omega |^{\frac{1}{r_{i}(p_{i}-1)}}L^{%
\frac{\max \{\gamma _{i},\theta _{i}\}}{p_{i}-1}}\right) ,%
\end{array}
\label{19*}
\end{equation}%
for some positive constants $k_{p_{i}}$ and $k_{0}$ depending only on $%
p_{i}, $ $M$ and $\Omega $. Bearing in mind (\ref{5}), \cite[Lemma]{LM}
applies, showing that the integral term $\int_{\Omega }d(x)^{r_{i}\mu _{i}}$ 
$\mathrm{d}x$ is bounded by some constant $\rho _{0}>0$. At this point, by (%
\ref{5}) and (\ref{19*}), we have%
\begin{equation*}
\Vert \nabla u_{i}\Vert _{\infty }\leq k_{0}(\rho _{0}^{\frac{1}{%
r_{i}(p_{i}-1)}}+|\Omega |^{\frac{1}{r_{i}(p_{i}-1)}}L^{\frac{\max \{\gamma
_{i},\theta _{i}\}}{p_{i}-1}})\leq L,
\end{equation*}%
provided $L>1$ is large enough. Hence, we conclude that $\mathcal{T}_{\infty
}(\mathcal{K}_{L}^{\infty })\subset \mathcal{K}_{L}^{\infty }$. The rest of
the proof runs as the proof of Theorem \ref{T1} showing that $%
(u_{1},u_{2})\in \mathcal{C}_{0}^{1,\sigma }(\overline{\Omega })\times 
\mathcal{C}_{0}^{1,\sigma }(\overline{\Omega }),$ $\sigma \in (0,1)$, is a
solution of problem $($\textrm{P}$)$ within $\left[ \underline{u}_{1},%
\overline{u}_{1}\right] \times \left[ \underline{u}_{2},\overline{u}_{2}%
\right] $.
\end{proof}

\section{Applications}

\label{S3}

\subsection{Convective system}

We consider the following convective system of quasilinear elliptic equations%
\begin{equation*}
(\mathrm{P}_{1})\qquad \left\{ 
\begin{array}{ll}
-\Delta _{p_{1}}u_{1}=u_{1}^{\alpha _{1}}u_{2}^{\beta _{1}}+|\nabla
u_{1}|^{\gamma _{1}}+(1+|\nabla u_{2}|)^{\theta _{1}} & \text{in }\Omega \\ 
-\Delta _{p_{2}}u_{2}=u_{1}^{\alpha _{2}}u_{2}^{\beta _{2}}+(1+|\nabla
u_{1}|)^{\gamma _{2}}+|\nabla u_{2}|^{\theta _{2}} & \text{in }\Omega \\ 
u_{1},u_{2}>0\text{ in }\Omega ,\text{ }u_{1},u_{2}=0\text{ on }\partial
\Omega , & 
\end{array}%
\right.
\end{equation*}%
where the exponents $\alpha _{i},\beta _{i},\gamma _{i}$ and $\theta _{i}$
are real constants such that 
\begin{equation}
p_{i}-1>|\alpha _{i}|+|\beta _{i}|\geq \alpha _{i}+\beta _{i}>-\frac{1}{r_{i}%
}  \label{alpha1}
\end{equation}
and%
\begin{equation}
\gamma _{2},\theta _{1}\leq 0\leq \gamma _{1},\theta _{2},\text{ \ }\gamma
_{1}<\frac{p_{1}-1}{r_{1}}\text{ \ and \ }\theta _{2}<\frac{p_{2}-1}{r_{2}}%
\text{ \ with }r_{i}\geq \frac{(p_{i}^{\ast })^{\prime }}{(p_{i})^{\prime }}.
\label{gamma1}
\end{equation}

Let consider $\xi _{i},\xi _{i,\delta }\in C^{1}(\overline{\Omega })$ the
unique solution of%
\begin{equation}
-\Delta _{p_{i}}\xi _{i}(x)=1+d(x)^{\alpha _{i}+\beta _{i}}\text{ \ in }%
\Omega ,\text{ }\xi _{i}(x)=0\text{ \ on }\partial \Omega ,  \label{7}
\end{equation}%
and%
\begin{equation}
-\Delta _{p_{i}}\xi _{i,\delta }(x)=\left\{ 
\begin{array}{ll}
1+d(x)^{\alpha _{i}+\beta _{i}} & \text{in }\Omega \backslash \overline{%
\Omega }_{\delta } \\ 
-1 & \text{in }\Omega _{\delta }%
\end{array}%
\right. ,\text{ }\xi _{i,\delta }(x)=0\text{ \ on }\partial \Omega ,
\label{9}
\end{equation}%
where $\Omega _{\delta }:=\{x\in \Omega :d(x)<\delta \}$, with a fixed $%
\delta >0$ sufficiently small.

Hardy-Sobolev inequality (see, e.g., \cite[Lemma 2.3]{AC}) guarantees that
the right-hand sides of (\ref{7}) and (\ref{9}) belong to $%
W^{-1,p_{i}^{\prime }}(\Omega )$. Consequently, Minty-Browder Theorem
ensures the existence of unique $\xi _{i}$ and $\xi _{i,\delta }$ in (\ref{9}%
) and (\ref{7}). Moreover, according to \cite[Lemma 3.1]{DM}, there are
positive constants $c_{0}$ and $c_{1}$ such that%
\begin{equation}
c_{0}d(x)\leq \xi _{i,\delta }(x)\leq \xi _{i}(x)\leq c_{1}d(x),\text{ \ for
all }x\in \Omega .  \label{4}
\end{equation}%
For $\alpha _{i}+\beta _{i}\geq 0$ (resp. $\alpha _{i}+\beta _{i}<0$ and $%
r_{i}>N$) in (\ref{alpha1}) and (\ref{gamma1}), owing to \cite[Lemma 1]{BE}
(resp. \cite[Corollary 2.2]{DM}), we get the gradient estimate%
\begin{equation}
||\nabla \xi _{i}||_{\infty }\leq \mathrm{M},\text{ for }i=1,2,  \label{4*}
\end{equation}%
for a certain constant $\mathrm{M}>0$ depending on $p_{1},p_{2},N,$ $\Omega $
and $||d(x)||_{\infty }.$ Here, it is important to observe from (\ref{alpha1}%
) that $(\alpha _{i}+\beta _{i})r_{i}>-1$ which, by \cite[Lemma, page 726]%
{LM}, guarantees%
\begin{equation*}
\int\limits_{\Omega }(1+d(x)^{\alpha _{i}+\beta _{i}})^{r_{i}}\text{ }%
\mathrm{d}x\leq k_{0}(1+\int\limits_{\Omega }d(x)^{(\alpha _{i}+\beta
_{i})r_{i}}\text{ }\mathrm{d}x)<\infty ,
\end{equation*}%
for some constant $k_{0}>0.$

Set%
\begin{equation}
(\underline{u}_{1},\underline{u}_{2})=C^{-1}(\xi _{1,\delta },\xi _{2,\delta
})\quad \text{and}\quad (\overline{u}_{1},\overline{u}_{2})=C(\xi _{1},\xi
_{2}),  \label{sub-super}
\end{equation}%
where $C>1$ is a constant. Obviously, $\underline{u}_{i}\leqslant \overline{u%
}_{i}$ in $\overline{\Omega },$ for $i=1,2.$

The following result allows us to achieve useful comparison properties.

\begin{proposition}
\label{P1}Assume (\ref{alpha1}) and (\ref{gamma1}) hold true. Then, for $C>1$
sufficiently large in (\ref{sub-super}), it holds%
\begin{equation*}
-\Delta _{p_{1}}\underline{u}_{1}\leq w_{1}^{\alpha _{1}}w_{2}^{\beta _{1}}%
\text{\ \ in }\Omega ,
\end{equation*}%
\begin{equation*}
-\Delta _{p_{2}}\underline{u}_{2}\leq w_{1}^{\alpha _{2}}w_{2}^{\beta _{2}}%
\text{\ \ in }\Omega ,
\end{equation*}%
\begin{equation*}
-\Delta _{p_{1}}\overline{u}_{1}\geq w_{1}^{\alpha _{1}}w_{2}^{\beta _{1}}+(%
\mathrm{M}C)^{\gamma _{1}}\text{\ \ in }\Omega ,
\end{equation*}%
\begin{equation*}
-\Delta _{p_{2}}\overline{u}_{2}\geq w_{1}^{\alpha _{2}}w_{2}^{\beta _{2}}+(%
\mathrm{M}C)^{\theta _{2}}\text{\ \ in }\Omega ,
\end{equation*}%
whenever $(w_{1},w_{2})\in \lbrack \underline{u}_{1},\overline{u}_{1}]\times
\lbrack \underline{u}_{2},\overline{u}_{2}]$.
\end{proposition}

\begin{proof}
On account of (\ref{4}) and (\ref{sub-super}), for all $(w_{1},w_{2})\in
\lbrack \underline{u}_{1},\overline{u}_{1}]\times \lbrack \underline{u}_{2},%
\overline{u}_{2}],$ we get 
\begin{equation}
\begin{array}{l}
w_{1}^{\alpha _{i}}w_{2}^{\beta _{i}}\geq \left\{ 
\begin{array}{ll}
\underline{u}_{1}^{\alpha _{i}}\underline{u}_{2}^{\beta _{i}} & \text{if }%
\alpha _{i},\beta _{i}\geq 0 \\ 
\overline{u}_{1}^{\alpha _{i}}\underline{u}_{2}^{\beta _{i}} & \text{if }%
\alpha _{i}\leq 0\leq \beta _{i} \\ 
\underline{u}_{1}^{\alpha _{i}}\overline{u}_{2}^{\beta _{i}} & \text{if }%
\beta _{i}\leq 0\leq \alpha _{i} \\ 
\overline{u}_{1}^{\alpha _{i}}\overline{u}_{2}^{\beta _{i}} & \text{if }%
\alpha _{i},\beta _{i}\leq 0%
\end{array}%
\right. \\ 
\geq \left\{ 
\begin{array}{ll}
(C^{-1}c_{0}d(x))^{\alpha _{i}+\beta _{i}} & \text{if }\alpha _{i},\beta
_{i}\geq 0 \\ 
(Cc_{1})^{\alpha _{i}}(C^{-1}c_{0})^{\beta _{i}}d(x)^{\alpha _{i}+\beta _{i}}
& \text{if }\alpha _{i}\leq 0\leq \beta _{i} \\ 
(C^{-1}c_{0})^{\alpha _{i}}(Cc_{1})^{\beta _{i}}d(x)^{\alpha _{i}+\beta _{i}}
& \text{if }\beta _{i}\leq 0\leq \alpha _{i} \\ 
(Cc_{1})^{\alpha _{i}+\beta _{i}}d(x)^{\alpha _{i}+\beta _{i}} & \text{if }%
\alpha _{i},\beta _{i}\leq 0%
\end{array}%
\right. \\ 
\geq \tilde{c}_{0}C^{-(|\alpha _{i}|+|\beta _{i}|)}d(x)^{\alpha _{i}+\beta
_{i}}\text{ \ in }\Omega%
\end{array}
\label{50}
\end{equation}%
and%
\begin{equation}
\begin{array}{l}
w_{1}^{\alpha _{i}}w_{2}^{\beta _{i}}\leq \left\{ 
\begin{array}{ll}
\overline{u}^{\alpha _{i}}\overline{v}^{\beta _{i}} & \text{if }\alpha
_{i},\beta _{i}\geq 0 \\ 
\underline{u}^{\alpha _{i}}\overline{v}^{\beta _{i}} & \text{if }\alpha
_{i}\leq 0\leq \beta _{i} \\ 
\overline{u}^{\alpha _{i}}\underline{v}^{\beta _{i}} & \text{if }\beta
_{i}\leq 0\leq \alpha _{i} \\ 
\underline{u}^{\alpha _{i}}\underline{v}^{\beta _{i}} & \text{if }\alpha
_{i},\beta _{i}\leq 0%
\end{array}%
\right. \\ 
\leq \left\{ 
\begin{array}{ll}
(Cc_{1}d(x))^{\alpha _{i}+\beta _{i}} & \text{if }\alpha _{i},\beta _{i}\geq
0 \\ 
(C^{-1}c_{0})^{\alpha _{i}}(Cc_{1})^{\beta _{i}}d(x)^{\alpha _{i}+\beta _{i}}
& \text{if }\alpha _{i}\leq 0\leq \beta _{i} \\ 
(Cc_{1})^{\alpha _{i}}(C^{-1}c_{0})^{\beta _{i}}d(x)^{\alpha _{i}+\beta _{i}}
& \text{if }\beta _{i}\leq 0\leq \alpha _{i} \\ 
(C^{-1}c_{0})^{\alpha _{i}}(C^{-1}c_{0})^{\beta _{i}}d(x)^{\alpha _{i}+\beta
_{i}} & \text{if }\alpha _{i},\beta _{i}\leq 0%
\end{array}%
\right. \\ 
\leq \tilde{c}_{1}C^{|\alpha _{i}|+|\beta _{i}|}d(x)^{\alpha _{i}+\beta _{i}}%
\text{ \ in }\Omega ,%
\end{array}%
\end{equation}%
where $\tilde{c}_{0},\tilde{c}_{1}$ are positive constants depending on $%
c_{0},c_{1},\alpha _{i},\beta _{i}$. From (\ref{alpha1}), (\ref{sub-super}),
(\ref{9}), (\ref{7}) and (\ref{4}), we get 
\begin{equation}
\begin{array}{l}
-\Delta _{p_{1}}\underline{u}=C^{-(p_{1}-1)}\left\{ 
\begin{array}{ll}
d(x)^{\alpha _{1}+\beta _{1}} & \text{in }\Omega \backslash \overline{\Omega 
}_{\delta } \\ 
-1 & \text{in }\Omega _{\delta }%
\end{array}%
\right. \\ 
\leq \tilde{c}_{0}m_{1}C^{-(|\alpha _{1}|+|\beta _{1}|)}d(x)^{\alpha
_{1}+\beta _{1}}\text{ in }\Omega ,%
\end{array}%
\end{equation}%
\begin{equation}
-\Delta _{p_{2}}\underline{v}=C^{-(p_{2}-1)}\left\{ 
\begin{array}{ll}
d(x)^{\alpha _{2}+\beta _{2}} & \text{in }\Omega \backslash \overline{\Omega 
}_{\delta } \\ 
-1 & \text{in }\Omega _{\delta }%
\end{array}%
\right. \leq \tilde{c}_{0}C^{-(|\alpha _{2}|+|\beta _{2}|)}d(x)^{\alpha
_{2}+\beta _{2}}\text{ in }\Omega ,
\end{equation}%
\begin{equation}
-\Delta _{p_{1}}\overline{u}_{1}=C^{p_{1}-1}\left( 1+d(x)^{\alpha _{1}+\beta
_{1}}\right) \geq \tilde{c}_{1}C^{|\alpha _{1}|+|\beta _{1}|}d(x)^{\alpha
_{1}+\beta _{1}}+(\mathrm{M}C)^{\gamma _{1}}\text{ \ in }\Omega
\end{equation}%
and%
\begin{equation}
-\Delta _{p_{2}}\overline{u}_{2}=C^{p_{2}-1}(1+d(x)^{\alpha _{2}+\beta
_{2}})\geq \tilde{c}_{1}C^{|\alpha _{2}|+|\beta _{2}|}d(x)^{\alpha
_{2}+\beta _{2}}+(\mathrm{M}C)^{\theta _{2}}\text{ \ in }\Omega ,  \label{51}
\end{equation}%
provided $C>1$ sufficiently large. Consequently, combining (\ref{50})-(\ref%
{51}), the desired estimates follow.
\end{proof}

\begin{theorem}
\label{T3}Assume that (\ref{alpha1}) and (\ref{gamma1}) are fulfilled. Then,
if $r_{i}\geq p_{i}^{\prime },$ system $(\mathrm{P}_{1})$ admits a
(positive) solution $(u_{1},u_{2})\in W_{0}^{1,p_{1}}(\Omega )\times
W_{0}^{1,p_{2}}(\Omega ),$ within $\left[ \underline{u}_{1},\overline{u}_{1}%
\right] \times \left[ \underline{u}_{2},\overline{u}_{2}\right] .$ Moreover,
if $r_{i}>N,$ this solution $(u_{1},u_{2})$ belongs to $\mathcal{C}%
_{0}^{1,\sigma }(\overline{\Omega })\times \mathcal{C}_{0}^{1,\sigma }(%
\overline{\Omega }),$ for certain $\sigma \in (0,1)$.
\end{theorem}

\begin{proof}
The idea is to apply Theorems \ref{T1} and \ref{T2} by showing that the
pairs $(\underline{u}_{1},\underline{u}_{2})$ and $(\overline{u}_{1},%
\overline{u}_{2})$ in (\ref{sub-super}) satisfy (\ref{sub}) and (\ref{sup}).
With this aim, let $(w_{1},w_{2})\in W_{0}^{1,p_{1}}(\Omega )\times
W_{0}^{1,p_{2}}(\Omega )$ such that $\underline{u}_{i}\leq w_{i}\leq 
\overline{u}_{i},$ $i=1,2.$ By Proposition \ref{P1}, it readily seen that%
\begin{equation*}
-\Delta _{p_{1}}\underline{u}_{1}\leq w_{1}^{\alpha _{1}}w_{2}^{\beta
_{1}}+|\nabla \underline{u}_{1}|^{\gamma _{1}}+(1+|\nabla w_{2}|)^{\theta
_{1}}\text{ in }\Omega ,
\end{equation*}%
\begin{equation*}
-\Delta _{p_{2}}\underline{u}_{2}\leq w_{1}^{\alpha _{2}}w_{2}^{\beta
_{2}}+(1+|\nabla w_{1}|)^{\gamma _{2}}+|\nabla \underline{u}_{2}|^{\theta
_{2}}\text{ in }\Omega ,
\end{equation*}%
and therefore, $(\underline{u}_{1},\underline{u}_{2}):=C^{-1}(\xi _{1,\delta
},\xi _{2,\delta })$ fulfills (\ref{sub}).

On account of (\ref{gamma1}) and (\ref{4*}), we get 
\begin{equation}
\begin{array}{l}
|\nabla \overline{u}_{1}|^{\gamma _{1}}+(1+|\nabla w_{2}|)^{\theta _{1}}\leq
1+|\nabla \overline{u}_{1}|^{\gamma _{1}} \\ 
\leq 1+(C||\nabla \xi _{1}||_{\infty })^{\gamma _{1}}\leq 1+(C\mathrm{M}%
)^{\gamma _{1}}\text{ in }\Omega%
\end{array}
\label{23}
\end{equation}%
and 
\begin{equation}
\begin{array}{l}
(1+|\nabla w_{2}|)^{\gamma _{2}}+|\nabla \overline{u}_{2}|^{\theta _{2}}\leq
1+|\nabla \overline{u}_{2}|^{\theta _{2}} \\ 
\leq 1+(C||\nabla \xi _{2}||_{\infty })^{\theta _{2}}\leq 1+(C\mathrm{M}%
)^{\theta _{2}}\text{ in }\Omega .%
\end{array}%
\end{equation}%
Then, again by Proposition \ref{P1}, we obtain%
\begin{equation*}
-\Delta _{p_{1}}\overline{u}_{1}\geq w_{1}^{\alpha _{1}}w_{2}^{\beta
_{1}}+|\nabla \overline{u}_{1}|^{\gamma _{1}}+(1+|\nabla w_{2}|)^{\theta
_{1}}\text{ in }\Omega
\end{equation*}%
and%
\begin{equation*}
-\Delta _{p_{2}}\overline{u}_{2}\geq w_{1}^{\alpha _{2}}w_{2}^{\beta
_{2}}+(1+|\nabla w_{1}|)^{\gamma _{2}}+|\nabla \overline{u}_{2}|^{\theta
_{2}}\text{ in }\Omega ,
\end{equation*}%
showing that (\ref{sup}) is fulfilled for $(\overline{u}_{1},\overline{u}%
_{2}):=C(\xi _{1},\xi _{2})$.

Consequently, depending on whether $r_{i}\geq p_{i}^{\prime }$ or $r_{i}>N$
in (\ref{alpha1})-(\ref{gamma1}), Theorems \ref{T1} and \ref{T2} provide a
solution $(u_{1},u_{2})$ in $W_{0}^{1,p_{1}}(\Omega )\times
W_{0}^{1,p_{2}}(\Omega )$ and $\mathcal{C}_{0}^{1,\sigma }(\overline{\Omega }%
)\times \mathcal{C}_{0}^{1,\sigma }(\overline{\Omega }),$ $\sigma \in (0,1),$
respectively, for problem $(\mathrm{P}_{1})$ such that $C^{-1}\xi _{i,\delta
}\leq u_{i}\leq C\xi _{i},$ $i=1,2.$ This ends the proof.
\end{proof}

\begin{remark}
Theorem \ref{T3} still valid if the gradient terms $(1+|\nabla
u_{1}|)^{\gamma _{2}}$ and $(1+|\nabla u_{2}|)^{\theta _{1}}$ are replaced
by bounded functions $h_{1}(\nabla u_{1})$ and $h_{2}(\nabla u_{2})$,
respectively.
\end{remark}

\subsection{Absorption system}

We state the following quasilinear absorption elliptic system%
\begin{equation*}
(\mathrm{P}_{2})\qquad \left\{ 
\begin{array}{l}
-\Delta _{p_{1}}u_{1}=u_{1}^{\alpha _{1}}u_{2}^{\beta _{1}}-|\nabla
u_{1}|^{\eta _{1}}\text{ \ in }\Omega \\ 
-\Delta _{p_{2}}u_{2}=u_{1}^{\alpha _{2}}u_{2}^{\beta _{2}}-|\nabla
u_{2}|^{\eta _{2}}\text{ \ in }\Omega \\ 
u_{1},u_{2}>0\text{ in }\Omega ,\text{ \ }u_{1},u_{2}=0\text{ on }\partial
\Omega ,%
\end{array}%
\right.
\end{equation*}%
which presents the singularities at the origin by assuming that the real
exponents $\alpha _{i},\beta _{i}$ satisfy (\ref{alpha1}) and $\eta _{i}$
fulfill 
\begin{equation}
|\alpha _{i}|+|\beta _{i}|\leq \eta _{i}<\frac{p_{i}-1}{r_{i}}\text{ \ with }%
r_{i}\geq \frac{(p_{i}^{\ast })^{\prime }}{(p_{i})^{\prime }}.
\label{gamma2}
\end{equation}%
Let $z_{i}\in C^{1}(\overline{\Omega })$ and $z_{i,\delta }\in C^{1}(%
\overline{\Omega })$ be the unique solution of%
\begin{equation}
-\Delta _{p_{i}}z_{i}(x)=d(x)^{\alpha _{i}+\beta _{i}}\text{ \ in }\Omega ,%
\text{ }z_{i}(x)=0\text{ \ on }\partial \Omega ,  \label{35*}
\end{equation}%
and%
\begin{equation}
-\Delta _{p_{i}}z_{i,\delta }(x)=\left\{ 
\begin{array}{ll}
d(x)^{\alpha _{i}+\beta _{i}} & \text{in }\Omega \backslash \overline{\Omega 
}_{\delta } \\ 
-1 & \text{in }\Omega _{\delta }%
\end{array}%
\right. ,\text{ }z_{i,\delta }(x)=0\text{ \ on }\partial \Omega ,  \label{35}
\end{equation}%
where $\Omega _{\delta }:=\{x\in \Omega :d(x)<\delta \}$, with a fixed $%
\delta >0$ sufficiently small. According to \cite[Lemma 3.1]{DM}, there are
positive constants $c_{2}$ and $c_{3}$ such that%
\begin{equation}
c_{2}d(x)\leq z_{i,\delta }(x)\leq z_{i}(x)\leq c_{3}d(x)\text{ \ for all }%
x\in \Omega ,  \label{36}
\end{equation}%
Moreover, for $\alpha _{i}+\beta _{i}\geq 0$ (resp. $\alpha _{i}+\beta
_{i}<0 $ and $r_{i}>N$) in (\ref{7}) and (\ref{alpha1}), owing to \cite[%
Lemma 1]{BE} (resp. \cite[Corollary 2.2]{DM}), we get the gradient estimate%
\begin{equation}
||\nabla z_{i}||_{\infty },||\nabla z_{i,\delta }||_{\infty }\leq \mathrm{%
\hat{M}},\text{ for }i=1,2,  \label{36*}
\end{equation}%
for a certain constant $\mathrm{\hat{M}}>0$ depending on $p_{1},p_{2},N$ and 
$\Omega .$

Set%
\begin{equation}
(\underline{u}_{1},\underline{u}_{2})=C^{-1}(z_{1,\delta },z_{2,\delta
})\quad \text{and}\quad (\overline{u}_{1},\overline{u}_{2})=C(z_{1},z_{2}),
\label{sub-super2}
\end{equation}%
where $C>1$ is a constant. Obviously, $\underline{u}_{i}\leqslant \overline{u%
}_{i}$ in $\overline{\Omega },$ for $i=1,2.$

The following result allows us to achieve useful comparison properties.

\begin{proposition}
\label{P2}Assume (\ref{alpha1}) and (\ref{gamma2}) hold true. Then, for $C>1$
sufficiently large in (\ref{sub-super2}), it holds%
\begin{equation}
-\Delta _{p_{i}}\underline{u}_{i}\leq w_{1}^{\alpha _{i}}w_{2}^{\beta
_{i}}-(C^{-1}\mathrm{\hat{M}})^{\eta _{i}}\text{\ \ in }\Omega ,  \label{55}
\end{equation}%
\begin{equation}
-\Delta _{p_{i}}\overline{u}_{i}\geq w_{1}^{\alpha _{i}}w_{2}^{\beta _{i}}%
\text{\ \ in }\Omega ,  \label{55*}
\end{equation}%
whenever $(w_{1},w_{2})\in \lbrack \underline{u}_{1},\overline{u}_{1}]\times
\lbrack \underline{u}_{2},\overline{u}_{2}],i=1,2$.
\end{proposition}

\begin{proof}
On the basis of (\ref{sub-super2}) and (\ref{36}), for all $(w_{1},w_{2})\in
\lbrack \underline{u}_{1},\overline{u}_{1}]\times \lbrack \underline{u}_{2},%
\overline{u}_{2}],$ we get 
\begin{equation}
\begin{array}{l}
w_{1}^{\alpha _{i}}w_{2}^{\beta _{i}}\geq \left\{ 
\begin{array}{ll}
1 & \text{if }\alpha _{i},\beta _{i}=0 \\ 
\overline{u}_{1}^{\alpha _{i}} & \text{if }\alpha _{i}<0=\beta _{i} \\ 
\overline{u}_{2}^{\beta _{i}} & \text{if }\beta _{i}<0=\alpha _{i} \\ 
\overline{u}_{1}^{\alpha _{i}}\overline{u}_{2}^{\beta _{i}} & \text{if }%
\alpha _{i},\beta _{i}<0%
\end{array}%
\right. \geq \left\{ 
\begin{array}{ll}
1 & \text{if }\alpha _{i},\beta _{i}=0 \\ 
(Cc_{3}d(x))^{\alpha _{i}} & \text{if }\alpha _{i}<0=\beta _{i} \\ 
(Cc_{3}d(x))^{\beta _{i}} & \text{if }\beta _{i}<0=\alpha _{i} \\ 
(Cc_{3}d(x))^{\alpha _{i}+\beta _{i}} & \text{if }\alpha _{i},\beta _{i}<0%
\end{array}%
\right. \\ 
\geq \tilde{c}_{1}C^{-(|\alpha _{i}|+|\beta _{i}|)}d(x)^{\alpha _{i}+\beta
_{i}}\text{ in }\Omega ,%
\end{array}
\label{42}
\end{equation}%
and%
\begin{equation}
\begin{array}{l}
w_{1}^{\alpha _{i}}w_{2}^{\beta _{i}}\leq \left\{ 
\begin{array}{ll}
1 & \text{if }\alpha _{i},\beta _{i}=0 \\ 
\underline{u}^{\alpha _{i}} & \text{if }\alpha _{i}<0=\beta _{i} \\ 
\underline{v}^{\beta _{i}} & \text{if }\beta _{i}<0=\alpha _{i} \\ 
\underline{u}^{\alpha _{i}}\underline{v}^{\beta _{i}} & \text{if }\alpha
_{i},\beta _{i}<0%
\end{array}%
\right. \\ 
\leq \left\{ 
\begin{array}{ll}
1 & \text{if }\alpha _{i},\beta _{i}=0 \\ 
(C^{-1}c_{2})^{\alpha _{i}}d(x)^{\alpha _{i}} & \text{if }\alpha
_{i}<0=\beta _{i} \\ 
(C^{-1}c_{2})^{\beta _{i}}d(x)^{\beta _{i}} & \text{if }\beta _{i}<0=\alpha
_{i} \\ 
(C^{-1}c_{2})^{\alpha _{i}+\beta _{i}}d(x)^{\alpha _{i}+\beta _{i}} & \text{%
if }\alpha _{i},\beta _{i}<0%
\end{array}%
\right. \leq \tilde{c}_{2}C^{|\alpha _{i}|+|\beta _{i}|}d(x)^{\alpha
_{i}+\beta _{i}},%
\end{array}
\label{45}
\end{equation}%
where $\tilde{c}_{2},\tilde{c}_{3}$ are positive constants depending on $%
\tilde{c}_{1},\tilde{c}_{2},\alpha _{i},\beta _{i}$. From (\ref{alpha1}) and
(\ref{gamma2}), one gets%
\begin{equation*}
C^{p_{i}-1}d(x)^{\alpha _{i}+\beta _{i}}\geq \tilde{c}_{2}C^{|\alpha
_{i}|+|\beta _{i}|}d(x)^{\alpha _{i}+\beta _{i}}\text{\ \ in}\ \Omega
\end{equation*}%
and 
\begin{equation}
\begin{array}{l}
C^{-(p_{i}-1)}d(x)^{\alpha _{i}+\beta _{i}}+(C^{-1}\mathrm{\hat{M}})^{\eta
_{i}} \\ 
\leq \max \{C^{-(p_{i}-1)},C^{-\eta _{i}}\}(d(x)^{\alpha _{i}+\beta _{i}}+%
\mathrm{\hat{M}}^{\eta _{i}}) \\ 
\leq \max \{C^{-(p_{i}-1)},C^{-\eta _{i}}\}(\delta ^{\alpha _{i}+\beta _{i}}+%
\mathrm{\hat{M}}^{\eta _{i}}) \\ 
\leq \tilde{c}_{i}C^{-(|\alpha _{i}|+|\beta _{i}|)}\delta ^{\alpha
_{i}+\beta _{i}}\leq \tilde{c}_{i}C^{-(|\alpha _{i}|+|\beta
_{i}|)}d(x)^{\alpha _{i}+\beta _{i}}\text{\ \ in}\ \Omega \backslash 
\overline{\Omega }_{\delta },%
\end{array}
\label{40}
\end{equation}%
while since 
\begin{equation*}
-C^{-(p_{i}-1)}\leq C^{-(p_{i}-1)}d(x)^{\alpha _{i}+\beta _{i}}\text{ in }%
\Omega _{\delta },
\end{equation*}%
we deduce that%
\begin{equation}
-C^{-(p_{i}-1)}+(C^{-1}\mathrm{\hat{M}})^{\eta _{i}}\leq \tilde{c}%
_{i}C^{-(|\alpha _{i}|+|\beta _{i}|)}d(x)^{\alpha _{i}+\beta _{i}}\text{\ \
in}\ \Omega _{\delta },  \label{41}
\end{equation}%
provided $C>1$ sufficiently large, for $i=1,2$. Then, gathering (\ref{35}), (%
\ref{sub-super2}), (\ref{42}), (\ref{40}) and (\ref{41}) together, we obtain%
\begin{equation*}
-\Delta _{p_{i}}\underline{u}_{i}=C^{-(p_{i}-1)}\left\{ 
\begin{array}{ll}
d(x)^{\alpha _{i}+\beta _{i}} & \text{in }\Omega \backslash \overline{\Omega 
}_{\delta } \\ 
-1 & \text{in }\Omega _{\delta }%
\end{array}%
\right. \leq w_{1}^{\alpha _{i}}w_{2}^{\beta _{i}}-(C^{-1}\mathrm{\hat{M}}%
)^{\eta _{i}}\text{ \ in\ }\Omega ,
\end{equation*}%
and%
\begin{equation*}
-\Delta _{p_{i}}\overline{u}_{i}=C^{p_{i}-1}d(x)^{\alpha _{i}+\beta
_{i}}\geq w_{1}^{\alpha _{i}}w_{2}^{\beta _{i}}\text{ \ in\ }\Omega ,
\end{equation*}%
for $(w_{1},w_{2})\in \lbrack \underline{u}_{1},\overline{u}_{1}]\times
\lbrack \underline{u}_{2},\overline{u}_{2}],i=1,2$. This shows (\ref{55})
and (\ref{55*}).
\end{proof}

\begin{theorem}
\label{T4}Assume that (\ref{alpha1}) and (\ref{gamma2}) are fulfilled. Then,
if $r_{i}\geq p_{i}^{\prime },$ system $(\mathrm{P}_{2})$ admits a
(positive) solution $(u_{1},u_{2})\in W_{0}^{1,p_{1}}(\Omega )\times
W_{0}^{1,p_{2}}(\Omega ),$ within $\left[ \underline{u}_{1},\overline{u}_{1}%
\right] \times \left[ \underline{u}_{2},\overline{u}_{2}\right] .$ Moreover,
if $r_{i}>N,$ this solution $(u_{1},u_{2})$ belongs to $\mathcal{C}%
_{0}^{1,\sigma }(\overline{\Omega })\times \mathcal{C}_{0}^{1,\sigma }(%
\overline{\Omega }),$ for certain $\sigma \in (0,1)$.
\end{theorem}

\begin{proof}
We shall prove that functions $(\underline{u}_{1},\underline{u}_{2})$ and $(%
\overline{u}_{1},\overline{u}_{2})$ in (\ref{sub-super}) satisfy (\ref{sub})
and (\ref{sup}), respectively. With this aim, let $(w_{1},w_{2})\in
W_{0}^{1,p_{1}}(\Omega )\times W_{0}^{1,p_{2}}(\Omega )$ such that 
\begin{equation*}
\underline{u}_{i}\leq w_{i}\leq \overline{u}_{i},\text{ }i=1,2.
\end{equation*}%
By Proposition \ref{P2}, it is readily seen that%
\begin{equation*}
-\Delta _{p_{i}}\overline{u}_{i}\geq w_{1}^{\alpha _{i}}w_{2}^{\beta
_{i}}\geq w_{1}^{\alpha _{i}}w_{2}^{\beta _{i}}-|\nabla \overline{u}%
_{i}|^{\eta _{i}}\text{ \ in\ }\Omega ,\text{ }i=1,2,
\end{equation*}%
and therefore, $(\overline{u}_{1},\overline{u}_{2}):=C(z_{1},z_{2})$
fulfills (\ref{sup}).

On account of (\ref{gamma2}), (\ref{36*}) and (\ref{sub-super}), we have 
\begin{equation}
|\nabla \underline{u}_{i}|^{\eta _{i}}=C^{-\eta _{i}}|\nabla z_{i,\delta
}|^{\eta _{i}}\leq C^{-\eta _{i}}||\nabla z_{i,\delta }||_{\infty }^{\eta
_{i}}\leq C^{-\eta _{i}}\mathrm{\hat{M}}^{\eta _{i}}\text{ \ in }\Omega .
\end{equation}%
Again by Proposition \ref{P2}, we obtain%
\begin{equation*}
-\Delta _{p_{i}}\underline{u}_{i}\leq w_{1}^{\alpha _{i}}w_{2}^{\beta
_{i}}-|\nabla \underline{u}_{i}|^{\eta _{i}}\text{ \ in\ }\Omega ,
\end{equation*}%
showing that (\ref{sub}) is fulfilled for $(\underline{u}_{1},\underline{u}%
_{2})=C^{-1}(z_{1,\delta },z_{2,\delta })$.

Consequently, depending on whether $r_{i}\geq p_{i}^{\prime }$ or $r_{i}>N$
in (\ref{alpha1}) and (\ref{gamma2}), Theorems \ref{T1} and \ref{T2} provide
a solution $(u_{1},u_{2})$ in $W_{0}^{1,p_{1}}(\Omega )\times
W_{0}^{1,p_{2}}(\Omega )$ and in $\mathcal{C}_{0}^{1,\sigma }(\overline{%
\Omega })\times \mathcal{C}_{0}^{1,\sigma }(\overline{\Omega }),$ for
certain $\sigma \in (0,1),$ respectively, for problem $(\mathrm{P}_{2})$
such that $C^{-1}z_{i,\delta }\leq u_{i}\leq Cz_{i},$ $i=1,2.$ This
completes the proof.
\end{proof}

\bigskip 

\textbf{Funding.} This research received no specific grant from any funding
agency in the public, commercial, or not-for-profit sectors.

\bigskip

\textbf{Conflict of Interest Statement.} The author has no competing interests to declare that are relevant to the content of this article.

\bigskip

\textbf{Data Availability Statement.} No datasets were generated or analyzed
during the current study.


\begin{thebibliography}{99}
\bibitem{AC} C.O. Alves \& F.J.S.A. Correa, \emph{On the existence of
positive solution for a class of singular systems involving quasilinear
operators}, Appl. Math. Comput. 185 (2007), 727-736.

\bibitem{ACF} C.O. Alves, P.C. Carri\~{a}o \& L.F.O. Faria, \emph{Existence
of solutions to singular elliptic equations with convection terms via the
Galerkin method}, Electron. J. Diff. Eqts. 12\ (2010), 1-12.

\bibitem{AM} C. O. Alves \& A. Moussaoui, \emph{Existence of solutions for a
class of singular elliptic systems with convection term}, Asymptot. Anal. 90
(2014), 237-248.

\bibitem{AM1} C.O. Alves \& A. Moussaoui, \emph{Existence and regularity of
solutions for a class of singular (}$\emph{p(x),q(x)}$\emph{)- Laplacian
system}, Complex Var. Elliptic Eqts. 63(2) (2017), 188-210.

\bibitem{AMT} C.O. Alves, A. Moussaoui \& L. Tavares, \emph{An elliptic
system with logarithmic nonlinearity}, Adv. Nonl. Anal. 8 (2019), 928-945.

\bibitem{Brezis} H. Brezis, \emph{Analyse Fonctionnelle Theorie et
Applications}. Masson, Paris (1983).

\bibitem{BE} H. Bueno \& G. Ercole, \emph{A quasilinear problem with fast
growing gradient}, Appl. Math. Lett., 26 (2013), 520-523.

\bibitem{CLMous} P. Candito, R. Livrea \& A.\ Moussaoui, \emph{Singular
quasilinear elliptic systems involving gradient terms}, Nonl. Anal.: Real
World App., 55 (2020), 103142.

\bibitem{CMo} S. Carl \& D. Motreanu, \emph{Extremal solutions for
nonvariational quasilinear elliptic systems via expanding trapping regions},
Monatsh Math 182 (2017), 801-821.

\bibitem{DM} H. Dellouche \& A. Moussaoui, \emph{Singular quasilinear
elliptic systems with gradient dependence}, Positivity, 26 (2022), 1-14.

\bibitem{DM1} H. Didi \& A. Moussaoui, \emph{Multiple positive solutions for
a class of quasilinear singular elliptic systems}, Rend. Circ. Mat. Palermo,
II. Ser 69 (2020), 977-994.

\bibitem{DM2} H. Didi, B. Khodja \& A. Moussaoui, \emph{Singular Quasilinear
Elliptic Systems With (super-) Homogeneous Condition}, J. Sibe. Fede. Univ.
Math. Phys. 13(2) (2020), 1-9.

\bibitem{DT} J.I. Diaz \& L. Tello, On a climate model with a dynamic
nonlinear diffusive boundary condition, Disc. Cont. Dyn. Syst. Serie S 1 (2)
2008, 253-262.

\bibitem{ElMiri} S. El-Hadi Miri,\emph{\ Quasilinear elliptic problems with
general growth and nonlinear term having singular behavior}, Advanced Nonl.
Studies 12 (2012), 19-48.

\bibitem{F} T.D. Frank, Nonlinear Fokker-Planck Equations Fundamentals and
Applications, Springer-Verlag, Berlin Heidelberg, 2005.

\bibitem{Hai} D. D. Hai, On a class of singular p-Laplacian boundary value
problems, J. Math. Anal. Appl. 383 (2011), 619-626.

\bibitem{GMM} U. Guarnotta, S.A. Marano \& A. Moussaoui, \emph{Singular
quasilinear convective elliptic systems in }$%
\mathbb{R}
^{N}$, Advances Nonl. Anal., Vol.11 (2022), 741-756.

\bibitem{K} M. Kaviany, Principles of Convective Heat Transfer. Springer,
New York, 2001.

\bibitem{KM} B. Khodja \& A. Moussaoui, \emph{Positive solutions for
infinite semipositone}$/$\emph{positone quasilinear elliptic systems with
singular and superlinear terms}, Diff. Eqts. App. 8(4) (2016), 535-546.

\bibitem{LM} A. C. Lazer \& P. J. Mckenna, \emph{On a singular nonlinear
elliptic boundary-value problem}, Proc. American Math. Soc. 3 (111) 1991,
721-730.

\bibitem{MMM} S.A. Marano, G. Marino \& A. Moussaoui, \emph{Singular
quasilinear elliptic systems in}\textit{\ }$%
\mathbb{R}
^{N}$, Ann. Mat. Pura Appl. 198(4) (2019), 1581-1594.

\bibitem{MVV} D. Motreanu, C. Vetro \& F. Vetro, \emph{Systems of
quasilinear elliptic equations with dependence on the gradient via
subsolution-supersolution method}, Disc. Cont. Dyn. Syst. Serie S 11 (2)
(2018), 309-321.

\bibitem{MM3} D. Motreanu \& A. Moussaoui, \emph{An existence result for a
class of quasilinear singular competitive elliptic systems}, Appl. Math.
Lett.\ 38 (2014), 33-37.

\bibitem{MM2} D. Motreanu \& A. Moussaoui, \emph{A quasilinear singular
elliptic system without} $\emph{cooperative}$ $\emph{structure}$, Acta Math.
Sci.\emph{\ }34 (B) (2014), 905-916.

\bibitem{MM1} D. Motreanu \& A. Moussaoui, \emph{Existence and boundedness
of solutions for a singular cooperative quasilinear elliptic system},
Complex Var. Elliptic Equ. 59 (2014), 285-296.

\bibitem{MMZ} D. Motreanu, A. Moussaoui \& Z. Zhang, \emph{Positive
solutions for singular elliptic systems with convection term}, J. Fixed
Point Theory Appl. 19 (3) (2017), 2165-2175.

\bibitem{MS} A. Moussaoui \& K. Saoudi, \emph{Existence and location of
nodal solutions for quasilinear convection-absorption Neumann problems, }%
Bull. Malays. Math. Sci. Soc. (47) 74 (2024), doi:
10.1007/s40840-024-01669-5.
\end{thebibliography}
\end{document}